\documentclass[11pt]{article}
\usepackage{tikz,caption,geometry,amsthm,amssymb,amsmath,enumerate,float,cite}
\newtheorem{theorem}{Theorem}

\newtheorem{corollary}[theorem]{Corollary}

\newtheorem{claim}{Claim}

\allowdisplaybreaks
\usepackage{lineno}
\usepackage{showlabels}

\title{On the maximum number of maximum independent sets}
\author{E. Mohr \and D. Rautenbach}
\date{}

\begin{document}

\maketitle

\begin{center}
Institut f\"{u}r Optimierung und Operations Research,
Universit\"{a}t Ulm, Ulm, Germany\\
\{\texttt{elena.mohr, dieter.rautenbach}\}\texttt{@uni-ulm.de}\\[3mm]
\end{center}

\newcommand{\mds}{\sharp\gamma}

\begin{abstract}
We give a very short and simple proof of Zykov's generalization 
of Tur\'{a}n's theorem,
which implies that 
the number of maximum independent sets
of a graph of order $n$ and independence number $\alpha$ 
with $\alpha<n$
is at most
$\left\lceil\frac{n}{\alpha}\right\rceil^{n\,{\rm mod}\,\alpha}
\left\lfloor\frac{n}{\alpha}\right\rfloor^{\alpha-(n\,{\rm mod}\,\alpha)}$.
Generalizing a result of Zito,
we show that 
the number of maximum independent sets
of a tree of order $n$ and independence number $\alpha$ is at most 
$2^{n-\alpha-1}+1$, if $2\alpha=n$, and,
$2^{n-\alpha-1}$, if $2\alpha>n$,
and we also characterize the extremal graphs.
Finally, we show that 
the number of maximum independent sets
of a subcubic tree of order $n$ and independence number $\alpha$ is at most 
$\left(\frac{1+\sqrt{5}}{2}\right)^{2n-3\alpha+1}$,
and we provide more precise results for extremal values of $\alpha$.
\end{abstract}
{\small
\begin{tabular}{lp{13.5cm}}
\textbf{Keywords:} & 
Tur\'{a}n graph; 
tree; 
independence number; 
maximum independent set;
Fibonacci number
\end{tabular}
}

\pagebreak

\section{Introduction}

We consider only finite, simple, and undirected graphs,
and use standard terminology and notation.
An {\it independent set} in a graph $G$
is a set of pairwise non-adjacent vertices of $G$.
The {\it independence number} $\alpha(G)$ of $G$ 
is the maximum cardinality of an independent set in $G$.
An independent set in $G$ is 
{\it maximal} if no proper superset is an independent set in $G$,
and 
{\it maximum} if it has cardinality $\alpha(G)$.
For a graph $G$, 
let $\sharp\alpha(G)$ be the number of maximum independent sets in $G$.

In the present paper 
we study the maximum number of maximum independent sets
as a function of the order and the independence number
in general graphs,
trees, 
and subcubic trees.
Before we come to our results, we mention some related research.

For a tree $T$ of order $n>1$, Zito \cite{zi} showed 
\begin{eqnarray}\label{e1}
\sharp\alpha(T)
& \leq &
\begin{cases}
2^{\frac{n-2}{2}}+1 & \mbox{, if $n$ is even, and}\\
2^{\frac{n-3}{2}} & \mbox{, if $n$ is odd.}
\end{cases}
\end{eqnarray}
Since $\alpha(T)\geq n/2$, it is not difficult to show that (\ref{e1}) implies
\begin{eqnarray}\label{e2}
\sharp\alpha(T) & \leq & 2^{\alpha(T)-1}+1,
\end{eqnarray}
cf.~\cite{aldamora} for a simple independent proof.
For similar results 
concerning the maximum number of maximal independent sets 
see \cite{kogodo,wl}.

Jou and Chang \cite{joch} observed that 
Moon and Moser's \cite{momo} result on the maximum number of maximal independent sets
implies
\begin{eqnarray*}
\sharp\alpha(G)
& \leq &
\begin{cases}
3^{\frac{n}{3}} & \mbox{, if $n\mod 3=0$,}\\
4\cdot 3^{\frac{n-4}{3}} & \mbox{, if $n\mod 3=1$, and}\\
2\cdot 3^{\frac{n-2}{3}} & \mbox{, if $n\mod 3=2$,}
\end{cases}
\end{eqnarray*}
for every graph $G$ of order $n$.
This is actually an immediate consequence of Zykov's generalization \cite{zy} 
of Tur\'{a}n's theorem \cite{tu};
independently shown also by Roman \cite{ro}.
For positive integers $n$ and $p$, let $T_p(n)$ be the complete $p$-partite graph 
with $n\,{\rm mod}\,p$ partite sets of order $\left\lceil\frac{n}{p}\right\rceil$
and $p-(n\,{\rm mod}\,p)$ partite sets of order $\left\lfloor\frac{n}{p}\right\rfloor$,
that is, $T_p(n)$ is the {\it Tur\'{a}n graph}.
A {\it clique} in a graph $G$
is a set of pairwise adjacent vertices of $G$.
For a graph $G$ and a positive integer $q$, 
let $\sharp\omega^{(p)}(G)$ be the number of cliques of order $p$ in $G$.

\begin{theorem}[Zykov \cite{zy}]\label{theoremzykov}
Let $n$, $q$, and $p$ be integers with $2\leq q<p\leq n$.
If $G$ is a graph of order $n$ with no clique of order $p$, then 
$\sharp\omega^{(q)}(G)\leq \sharp\omega^{(q)}\left(T_{p-1}(n)\right)$
with equality if and only if $G=T_{p-1}(n)$.
\end{theorem}
As our first contribution, 
we give a very short and simple proof of Theorem \ref{theoremzykov}
inspired by the 5th {\it proof from The Book} \cite{aizi} of Tur\'{a}n's theorem.
Applying the special case $q=p-1$ of Theorem \ref{theoremzykov} 
to the complement $\bar{G}$ of a graph $G$
immediately implies the following.

\begin{corollary}\label{corollary1}
If $G$ is a graph of order $n$ and independence number $\alpha$ 
with $\alpha<n$, then
\begin{eqnarray}\label{ecor}
\sharp\alpha(G)\leq 
\left\lceil\frac{n}{\alpha}\right\rceil^{n\,{\rm mod}\,\alpha}
\left\lfloor\frac{n}{\alpha}\right\rfloor^{\alpha-(n\,{\rm mod}\,\alpha)}.
\end{eqnarray}
Furthermore, equality holds in (\ref{ecor})
if and only if $G$ is the complement of $T_{\alpha}(n)$.
\end{corollary}
Corollary \ref{corollary1} also follows from a result of Nielsen \cite{ni} 
who showed that the right hand side of (\ref{ecor})
is a tight upper bound on the number 
of maximal independent sets of cardinality exactly $\alpha$
for every graph $G$ of order $n$ regardless of the independence number of $G$.

Our further results concern trees and subcubic trees.

The next result is a common generalization of (\ref{e1}) and (\ref{e2}).

\begin{theorem}\label{theorem1}
If $T$ is a tree of order $n$ and independence number $\alpha$, then 
\begin{eqnarray}\label{e4}
\sharp\alpha(T)
& \leq &
\begin{cases}
2^{n-\alpha-1}+1 & \mbox{, if $2\alpha=n$, and}\\
2^{n-\alpha-1} & \mbox{, if $2\alpha>n$.}
\end{cases}
\end{eqnarray}
Furthermore, equality holds in (\ref{e4})
if and only if $T$ arises by subdividing $n-\alpha-1$ edges of $K_{1,\alpha}$ once.
\end{theorem}
As it turns out, 
the maximum number of maximum independent sets in subcubic trees
is closely related to the famous Fibonacci numbers.
Let $f(n)$ denote the $n$-th {\it Fibonacci number}, that is, 
$$f(n)=
\begin{cases}
0 & \mbox{, if $n=0$,}\\
1 & \mbox{, if $n=1$,}\\
f(n-1)+f(n-2) & \mbox{, if $n\geq 2$.}
\end{cases}$$
Our first result for subcubic trees
concerns the smallest possible value of the independence number
in (subcubic) trees.
For a positive integer $k$, let $T(k)$ arise by attaching a new endvertex
to every vertex of a path of order $k$.
Since 
\begin{eqnarray*}
\sharp\alpha(T(1))&=&2, \\
\sharp\alpha(T(2))&=&3,\mbox{ and}\\
\sharp\alpha(T(k))&=&\sharp\alpha(T(k-1))+\sharp\alpha(T(k-2))\mbox{ for every $k\geq 3$},
\end{eqnarray*}
we obtain
$$\sharp\alpha(T(k))=f(k+2)$$ 
for every positive integer $k$.
\begin{theorem}\label{theorem2}
If $T$ is a subcubic tree of order $n$ and independence number $\alpha=\frac{n}{2}$, then 
\begin{eqnarray}\label{e9}
\sharp\alpha(T)
& \leq & f(\alpha+2)
\end{eqnarray}
Furthermore, equality holds in (\ref{e9})
if and only if $T=T(\alpha)$.
\end{theorem}
Our second result for subcubic trees
concerns the largest possible value of the independence number
in subcubic trees.
If $T$ is a tree, 
then $T'$ {\it arises from $T$ by attaching a $P_3$}
if $V(T)$ is the disjoint union of $V(T)$ and $\{ x,y,z\}$,
and $E(T)=E(T')\cup \{ uy,xy,yz\}$,
where $u$ is some vertex of $T$.

\begin{theorem}\label{theorem3}
If $T$ is a subcubic tree of order $n$ and independence number $\alpha$, 
then 
\begin{eqnarray}\label{e11}
\alpha(T)
& \leq & \frac{2n+1}{3}.
\end{eqnarray}
Furthermore, equality holds in (\ref{e11})
if and only if $T$ arises from $K_1$ by iteratively attaching $P_3$s,
in which case $\sharp\alpha(T)=1$.
\end{theorem}
For given positive integers $n$ and $\alpha$ with $\alpha\leq \frac{2n+1}{3}$,
suitably combining the extremal trees 
from Theorem \ref{theorem2}
and Theorem \ref{theorem3} allows to construct subcubic trees
with order $n$ and independence number $\alpha$
that satisfy
$$\sharp\alpha(T)=\Omega\Big(f(2n-3\alpha+1)\Big).$$
This implies that our last result for subcubic trees
is best possible 
up to small constant factors and additive terms.

\begin{theorem}\label{theorem4}
If $T$ is a subcubic tree of order $n$ and independence number $\alpha$, then $$\sharp\alpha(T)\leq \left(\frac{1+\sqrt{5}}{2}\right)^{2n-3\alpha+1}.$$
\end{theorem}
All proofs are give in the next section.

\section{Proofs}

\begin{proof}[Proof of Theorem \ref{theoremzykov}]
Let $G$ be a graph of order $n$ with no clique of order $p$ 
that maximizes $\sharp\omega^{(q)}(G)$. 
Let $G_0$ arise from $G$ by removing all edges 
that do not belong to a clique of order $q$ in $G$.
Clearly, $G_0$ has no clique of order $p$, 
and $\sharp\omega^{(q)}(G_0)=\sharp\omega^{(q)}(G)$.

\begin{claim}\label{claim1}
$G_0$ is a complete multipartite graph.
\end{claim}
\begin{proof}[Proof of Claim \ref{claim1}]
Suppose, for a contradiction, that the claim fails.
This implies the existence of three vertices $u$, $v$, and $w$
such that $u$ is not adjacent to $v$ or $w$, but $v$ and $w$ are adjacent.
Let $d^{(q)}(u)$ be the number of cliques of order $q$ in $G_0$ that contain $u$,
that is, $d^{(q)}(u)=\sharp\omega^{(q-1)}(G_0[N_{G_0}(u)])$.
Let $d^{(q)}(v)$ and $d^{(q)}(w)$ be defined analogously.
If $d^{(q)}(u)<d^{(q)}(v)$, then 
the graph that arises from $G_0$ by removing $u$ and duplicating $v$
has no clique of order $p$ but 
$\sharp\omega^{(q)}(G_0)-d^{(q)}(u)+d^{(q)}(v)>\sharp\omega^{(q)}(G)$ 
cliques of order $q$, contradicting the choice of $G$.
Hence, by symmetry, we may assume that $d^{(q)}(u)\geq d^{(q)}(v),d^{(q)}(w)$.
Now, since the edge $vw$ belongs to some clique of order $q$ in $G_0$,
the graph that arises from $G_0$ by removing $v$ and $w$, and triplicating $u$
has no clique of order $p$ but 
$\sharp\omega^{(q)}(G_0)+2d^{(q)}(u)-d^{(q)}(v)-d^{(q)}(w)+1>\sharp\omega^{(q)}(G)$ 
cliques of order $q$, contradicting the choice of $G$.
\end{proof}
Since $G_0$ has no clique of order $p$,
the multipartite graph $G_0$ has $p-1$ (possibly empty) partite sets $V_1,\ldots,V_{p-1}$,
of orders $n_1\geq \ldots\geq n_{p-1}$, respectively.
Since $\sharp\omega^{(q)}(G_0)>0$,
the graph $G_0'=G_0-(V_1\cup V_{p-1})$ has a clique of order $q-2$,
that is, $\sharp\omega^{(q-2)}(G_0')>0$.
If $n_1\geq n_{p-1}+2$, then $G_0$ has
$$
n_1n_{p-1}\sharp\omega^{(q-2)}(G_0')
+
(n_1+n_{p-1})\sharp\omega^{(q-1)}(G_0')
+
\sharp\omega^{(q)}(G_0')$$
cliques of order $q$,
while 
the graph that arises from $G_0$ by moving one vertex from $V_i$ to $V_j$
has 
$$
(n_1-1)(n_{p-1}+1)\sharp\omega^{(q-2)}(G_0')
+
(n_1-1+n_{p-1}+1)\sharp\omega^{(q-1)}(G_0')
+
\sharp\omega^{(q)}(G_0')$$
cliques of order $q$.
Since 
$\sharp\omega^{(q-2)}(G_0')>0$
and 
$(n_1-1)(n_{p-1}+1)>n_1n_{p-1}$,
this contradicts the choice of $G$.
Hence, we obtain $|n_i-n_j|\leq 1$ for every $1\leq i\leq j\leq p-1$,
which implies $G_0=T_{p-1}(n)$.
Since $n\geq p$, all $p-1$ partite sets of $G_0$ are non-empty.
Therefore, adding any non-edge of $G_0$ to $G_0$ 
results in a graph that has a clique of order $p$,
which implies $G=G_0$,
and completes the proof. 
\end{proof}
A vertex of degree at most $1$ is an {\it endvertex},
and a neighbor of an endvertex is a {\it support vertex}.

\begin{proof}[Proof of Theorem \ref{theorem1}]
Within this proof, we call a tree {\it special} if it arises by subdividing $n-\alpha-1$ edges of $K_{1,\alpha}$ once.
Suppose, for a contradiction, that the theorem is false,
and let $n$ be the smallest order for which it fails.
Let $T$ be a tree of order $n$ and independence number $\alpha$ 
such that 
\begin{itemize}
\item either $\sharp\alpha(T)$ does not satisfy (\ref{e4}),
\item or $\sharp\alpha(T)$ satisfies (\ref{e4}) with equality but $T$ is not special.
\end{itemize}
It is easy to see that $T$ is not special and has diameter at least $3$,
which implies $\frac{n}{2}\leq \alpha\leq n-2$.
We root $T$ at an endvertex of a longest path in $T$.
Let $y$ be the parent of an endvertex of maximum depth in $T$, 
let $x_1,\ldots,x_k$ be the children of $y$, and
let $z$ be the parent of $y$.

The tree $T'=T-\{ x_1,\ldots,x_k,y\}$ has order $n'=n-k-1$ and independence number $\alpha'=\alpha-k$.

First, we assume that $k\geq 2$.
In this case, every maximum independent set in $T$ contains $\{ x_1,\ldots,x_k\}$, 
and the choice of $n$ implies
\begin{eqnarray}
\sharp\alpha(T) & = & \sharp\alpha(T')\nonumber\\
& \stackrel{(\ref{e4})}{\leq} & 2^{n'-\alpha'-1}+1\label{e5}\\
& = & 2^{n-\alpha-2}+1\nonumber\\
& \stackrel{\alpha\leq n-2}{\leq} & 2^{n-\alpha-1}\label{e6}.
\end{eqnarray}
Now, if $\sharp\alpha(T)=2^{n-\alpha-1}$,
then 
\begin{itemize}
\item equality holds in (\ref{e5}), which implies $2(\alpha-k)=2\alpha'=n'=n-k-1$,
and 
\item equality holds in (\ref{e6}), which implies $\alpha=n-2$.
\end{itemize}
These equations imply $k=n-3$, $\alpha'=1$, and $n'=2$, that is, $T'$ is $K_2$.
We obtain the contradiction, that $T$ arises by sudvidiving one edge of $K_{1,\alpha}$,
that is, $T$ is special.
Hence, we may assume that $k=1$.

Since 
the number of maximum independent sets in $T$ that contain $y$
is less or equal than 
the number of maximum independent sets in $T$ that contain $x$,
we obtain $\sharp\alpha(T) \leq 2\sharp\alpha(T')$,
and $\sharp\alpha(T)<2\sharp\alpha(T')$ if some maximum independent set in $T'$ that contain $z$.

First, we assume that $2\alpha=n$ and that $T'$ is not special.
Since $2\alpha'=2\alpha-2=n-2=n'$,
the tree $T'$ is a bipartite graph whose partite sets both have order exactly $\alpha'$.
This implies that some maximum independent set in $T'$ contains $z$,
and the choice of $n$ implies the contradiction
\begin{eqnarray*}
\sharp\alpha(T) & < & 2\sharp\alpha(T')
\stackrel{(\ref{e4})}{\leq} 2\cdot 2^{n'-\alpha'-1}
=2^{n-\alpha-1}.
\end{eqnarray*}
Next, we assume that $2\alpha=n$ and that $T'$ is special.
There are only three possibilities for the structure of $T$
illustrated in Figure \ref{fig:3cases} 
together with the resulting values of $\sharp\alpha$.

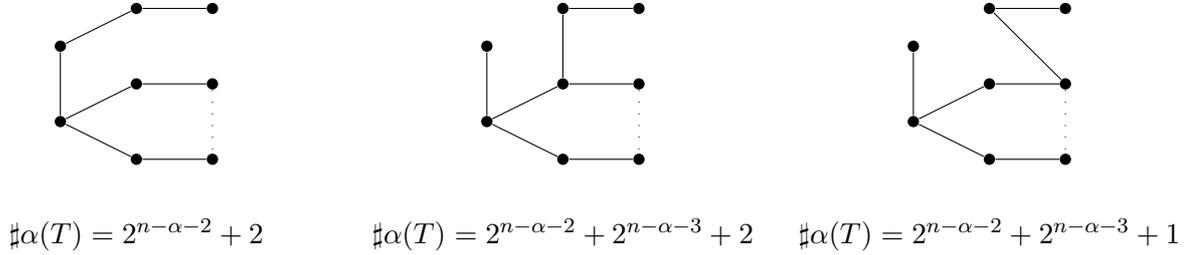
\begin{figure} [H]
\begin{minipage}[t]{0.33\textwidth}
\begin{center}
\begin{tikzpicture}[auto]

 \node[fill, circle, inner sep=1.5pt] (v1) at (0,0.5) {};
 \node[fill, circle, inner sep=1.5pt] (v2) at (1,0) {};
 \node[fill, circle, inner sep=1.5pt] (v3) at (2,0) {};
 \node[fill, circle, inner sep=1.5pt] (v4) at (1,2) {};
 \node[fill, circle, inner sep=1.5pt] (v5) at (1,1) {};
 \node[fill, circle, inner sep=1.5pt] (v6) at (2,1) {};
 \node[fill, circle, inner sep=1.5pt] (v7) at (0,1.5) {};
 \node[fill, circle, inner sep=1.5pt] (v8) at (2,2) {};
 
\draw (v5) -- (v1);
\draw (v1) -- (v2);
\draw (v1) -- (v7);
\draw (v5) -- (v6);
\draw (v2) -- (v3);
\draw (v7) -- (v4);
\draw (v4) -- (v8);
\draw[loosely dotted] (v6)--(v3);

\node[align=left] at (1,-1) {$\sharp\alpha(T)=2^{n-\alpha-2}+2$};

\end{tikzpicture}
\end{center}

\end{minipage}\begin{minipage}[t]{0.33\textwidth}
\begin{center}
\begin{tikzpicture}[auto]

 \node[fill, circle, inner sep=1.5pt] (v1) at (0,0.5) {};
 \node[fill, circle, inner sep=1.5pt] (v2) at (1,0) {};
 \node[fill, circle, inner sep=1.5pt] (v3) at (2,0) {};
 \node[fill, circle, inner sep=1.5pt] (v4) at (1,2) {};
 \node[fill, circle, inner sep=1.5pt] (v5) at (1,1) {};
 \node[fill, circle, inner sep=1.5pt] (v6) at (2,1) {};
 \node[fill, circle, inner sep=1.5pt] (v7) at (0,1.5) {};
 \node[fill, circle, inner sep=1.5pt] (v8) at (2,2) {};
 
\draw (v5) -- (v1);
\draw (v1) -- (v2);
\draw (v1) -- (v7);
\draw (v5) -- (v6);
\draw (v2) -- (v3);
\draw (v5) -- (v4);
\draw (v4) -- (v8);
\draw[loosely dotted] (v6)--(v3);

\node[align=left] at (1,-1) {$\sharp\alpha(T)=2^{n-\alpha-2}+2^{n-\alpha-3}+2$};
\end{tikzpicture}
\end{center}
\end{minipage}\begin{minipage}[t]{0.33\textwidth}
\begin{center}
\begin{tikzpicture}[auto]

 \node[fill, circle, inner sep=1.5pt] (v1) at (0,0.5) {};
 \node[fill, circle, inner sep=1.5pt] (v2) at (1,0) {};
 \node[fill, circle, inner sep=1.5pt] (v3) at (2,0) {};
 \node[fill, circle, inner sep=1.5pt] (v4) at (1,2) {};
 \node[fill, circle, inner sep=1.5pt] (v5) at (1,1) {};
 \node[fill, circle, inner sep=1.5pt] (v6) at (2,1) {};
 \node[fill, circle, inner sep=1.5pt] (v7) at (0,1.5) {};
 \node[fill, circle, inner sep=1.5pt] (v8) at (2,2) {};
 
\draw (v5) -- (v1);
\draw (v1) -- (v2);
\draw (v1) -- (v7);
\draw (v5) -- (v6);
\draw (v2) -- (v3);
\draw (v6) -- (v4);
\draw (v4) -- (v8);
\draw[loosely dotted] (v6)--(v3);

\node[align=left] at (1,-1) {$\sharp\alpha(T)=2^{n-\alpha-2}+2^{n-\alpha-3}+1$};

\end{tikzpicture}
\end{center}
\end{minipage}
\captionof{figure}{Three possibilities for the structure of $T$.}
\label{fig:3cases}
\end{figure}
In all three cases, we have $n-\alpha-2\geq 1$,
because otherwise either $T$ would be special or 
the configuration would not be possible. 
In the first and third case,
this already implies a contradiction, because
$2^{n-\alpha -2}+2\leq 2^{n-\alpha -2}+2^{n-\alpha -3}+1 \leq2^{n-\alpha -1}.$
In the second case, 
we obtain $n-\alpha-2\geq 2$, because $T$ is not special. 
Thus, also in this case, we obtain a contradiction, because
$2^{n-\alpha -2}+2^{n-\alpha -3}+2 \leq2^{n-\alpha -1}.$

Finally, we assume that $2\alpha>n$.
Since $2\alpha'>n'$, the choice of $n$ implies
\begin{eqnarray}
\sharp\alpha(T) & \leq & 2\sharp\alpha(T')\label{e7}\\
&\stackrel{(\ref{e4})}{\leq} & 2\cdot 2^{n'-\alpha'-1}\label{e8}\\
&=&2^{n-\alpha-1}.\nonumber
\end{eqnarray}
Now, if $\sharp\alpha(T)=2^{n-\alpha-1}$,
then 
\begin{itemize}
\item equality holds in (\ref{e7}), which implies that no maximum independent set in $T'$ contains $z$,
and 
\item equality holds in (\ref{e8}), which implies that $T'$ is special.
\end{itemize}
Since the only vertex of $T'$ that does not belong to some maximum independent set in $T'$
is the unique vertex of degree more than $2$ in $T'$,
we obtain the contradiction that $T$ is special,
which completes the proof.
\end{proof}

\begin{proof}[Proof of Theorem \ref{theorem2}]
Suppose, for a contradiction, that the theorem is false,
and let $n$ be the smallest order for which it fails.
Let $T$ be a subcubic tree of order $n$ and independence number $\alpha=\frac{n}{2}$
such that $\sharp\alpha(T)$ is as large as possible.
Note that $n$ is necessarily even.

If $A$ and $B$ are the two partite sets of the bipartite graph $T$,
then $\alpha=\frac{n}{2}$ implies $|A|=|B|=\frac{n}{2}$.
Furthermore, since $A$ and $B$ are both maximum independent sets in $T$,
the neighborhood $N_T(S)$ of every subset $S$ of $A$ is at least as large as $S$,
which, by Hall's theorem \cite{ha}, implies that $T$ has a perfect matching $M$.
If $n\in \{ 2,4\}$, then $T=T(\alpha)$ follows immediately.
Hence, we may assume that $n\geq 6$.

Let the tree $\tilde{T}$ arise from $T$ by contracting all edges in $M$.
Let $e_1\ldots e_p$ be a longest path in $\tilde{T}$.
Since $n\geq 6$, we have $p\geq 3$.
Let $e_i=u_iv_i$ for $i\in [3]$.
By symmetry, we may assume that 
$u_2u_3$ is the (unique) edge between $e_2$ and $e_3$.
By the choice of $P$, all neighbors of $e_2$ in $\tilde{T}$ that are distinct from $e_3$ are endvertices of $\tilde{T}$.
Since $T$ has maximum degree at most $3$,
the set $N_{\tilde{T}}(e_2)\setminus \{ e_3\}$
contains 
\begin{itemize}
\item $d_1\leq 1$ edges $e$ of $T$ such that $u_2$ has a neighbor in $e$,
and 
\item $d_2\leq 2$ edges $e$ of $T$ such that $v_2$ has a neighbor in $e$.
\end{itemize}
Since $e_1$ is one of the edges counted by $d_1+d_2$,
we obtain
$$(d_1,d_2)\in \{ (0,1),(0,2),(1,1),(1,2),(1,0)\}.$$

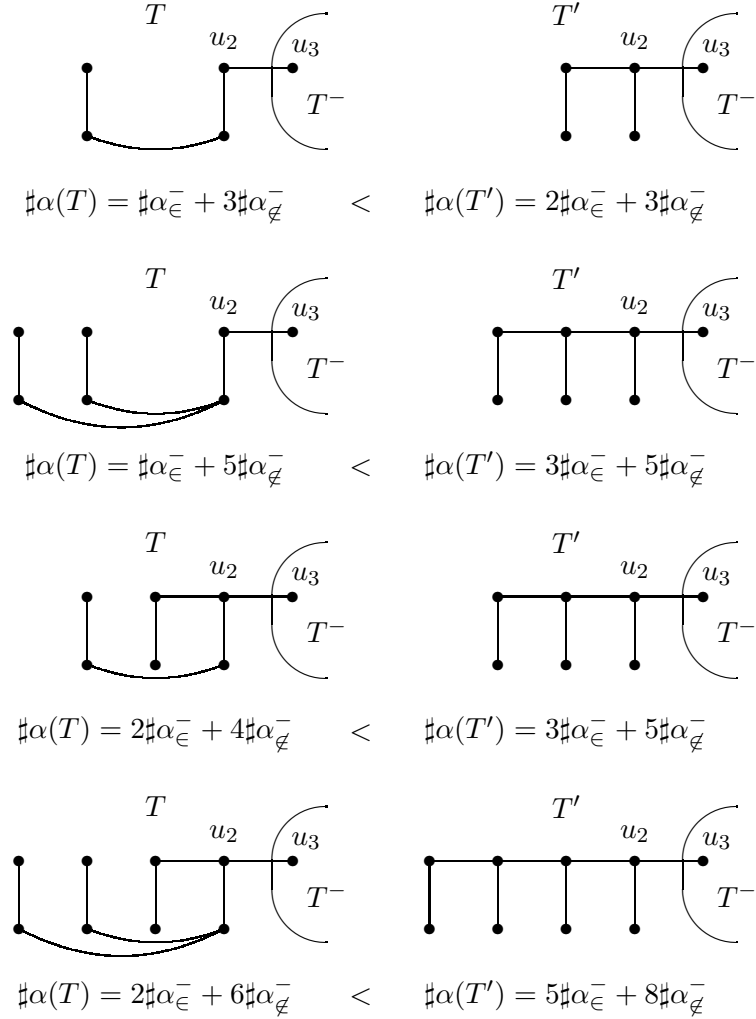
\begin{figure}[H]
\begin{center}
\unitlength 0.9mm 
\linethickness{0.4pt}
\ifx\plotpoint\undefined\newsavebox{\plotpoint}\fi 
\begin{picture}(118,33)(0,0)
\put(15,15){\circle*{1.5}}
\put(85,15){\circle*{1.5}}
\put(35,15){\circle*{1.5}}
\put(95,15){\circle*{1.5}}
\put(15,25){\circle*{1.5}}
\put(85,25){\circle*{1.5}}
\put(35,25){\circle*{1.5}}
\put(95,25){\circle*{1.5}}
\put(45,25){\circle*{1.5}}
\put(105,25){\circle*{1.5}}
\put(15,25){\line(0,-1){10}}
\put(85,25){\line(0,-1){10}}
\put(35,25){\line(0,-1){10}}
\put(95,25){\line(0,-1){10}}
\put(47,28){\makebox(0,0)[cc]{$u_3$}}
\put(107,28){\makebox(0,0)[cc]{$u_3$}}
\put(35,29){\makebox(0,0)[cc]{$u_2$}}
\put(95,29){\makebox(0,0)[cc]{$u_2$}}
\put(50,23){\oval(16,20)[l]}
\put(110,23){\oval(16,20)[l]}
\put(50,20){\makebox(0,0)[cc]{$T^-$}}
\put(110,20){\makebox(0,0)[cc]{$T^-$}}
\put(45,25){\line(-1,0){10}}
\put(105,25){\line(-1,0){10}}
\put(95,25){\line(-1,0){10}}
\qbezier(35,15)(25,11)(15,15)
\put(25,5){\makebox(0,0)[cc]
{$\sharp\alpha(T)=\sharp\alpha^-_\in+3\sharp\alpha^-_{\not\in}$}}
\put(85,5){\makebox(0,0)[cc]
{$\sharp\alpha(T')=2\sharp\alpha^-_\in+3\sharp\alpha^-_{\not\in}$}}
\put(55,5){\makebox(0,0)[cc]{$<$}}
\put(25,33){\makebox(0,0)[cc]{$T$}}
\put(85,33){\makebox(0,0)[cc]{$T'$}}
\end{picture}\\[5mm]
\linethickness{0.4pt}
\ifx\plotpoint\undefined\newsavebox{\plotpoint}\fi 
\begin{picture}(118,33)(0,0)
\put(5,15){\circle*{1.5}}
\put(15,15){\circle*{1.5}}
\put(75,15){\circle*{1.5}}
\put(85,15){\circle*{1.5}}
\put(35,15){\circle*{1.5}}
\put(95,15){\circle*{1.5}}
\put(5,25){\circle*{1.5}}
\put(15,25){\circle*{1.5}}
\put(75,25){\circle*{1.5}}
\put(85,25){\circle*{1.5}}
\put(35,25){\circle*{1.5}}
\put(95,25){\circle*{1.5}}
\put(45,25){\circle*{1.5}}
\put(105,25){\circle*{1.5}}
\put(5,25){\line(0,-1){10}}
\put(15,25){\line(0,-1){10}}
\put(75,25){\line(0,-1){10}}
\put(85,25){\line(0,-1){10}}
\put(35,25){\line(0,-1){10}}
\put(95,25){\line(0,-1){10}}
\put(47,28){\makebox(0,0)[cc]{$u_3$}}
\put(107,28){\makebox(0,0)[cc]{$u_3$}}
\put(35,29){\makebox(0,0)[cc]{$u_2$}}
\put(95,29){\makebox(0,0)[cc]{$u_2$}}
\put(50,23){\oval(16,20)[l]}
\put(110,23){\oval(16,20)[l]}
\put(50,20){\makebox(0,0)[cc]{$T^-$}}
\put(110,20){\makebox(0,0)[cc]{$T^-$}}
\put(45,25){\line(-1,0){10}}
\put(105,25){\line(-1,0){10}}
\put(95,25){\line(-1,0){10}}
\qbezier(35,15)(25,11)(15,15)
\qbezier(5,15)(20,7)(35,15)
\put(25,5){\makebox(0,0)[cc]
{$\sharp\alpha(T)=\sharp\alpha^-_\in+5\sharp\alpha^-_{\not\in}$}}
\put(85,5){\makebox(0,0)[cc]
{$\sharp\alpha(T')=3\sharp\alpha^-_\in+5\sharp\alpha^-_{\not\in}$}}
\put(55,5){\makebox(0,0)[cc]{$<$}}
\put(25,33){\makebox(0,0)[cc]{$T$}}
\put(85,33){\makebox(0,0)[cc]{$T'$}}
\put(75,25){\line(1,0){10}}
\end{picture}\\[5mm]
\linethickness{0.4pt}
\ifx\plotpoint\undefined\newsavebox{\plotpoint}\fi 
\begin{picture}(118,33)(0,0)
\put(15,15){\circle*{1.5}}
\put(75,15){\circle*{1.5}}
\put(25,15){\circle*{1.5}}
\put(85,15){\circle*{1.5}}
\put(35,15){\circle*{1.5}}
\put(95,15){\circle*{1.5}}
\put(15,25){\circle*{1.5}}
\put(75,25){\circle*{1.5}}
\put(25,25){\circle*{1.5}}
\put(85,25){\circle*{1.5}}
\put(35,25){\circle*{1.5}}
\put(95,25){\circle*{1.5}}
\put(45,25){\circle*{1.5}}
\put(105,25){\circle*{1.5}}
\put(15,25){\line(0,-1){10}}
\put(75,25){\line(0,-1){10}}
\put(25,25){\line(0,-1){10}}
\put(85,25){\line(0,-1){10}}
\put(35,25){\line(0,-1){10}}
\put(95,25){\line(0,-1){10}}
\put(47,28){\makebox(0,0)[cc]{$u_3$}}
\put(107,28){\makebox(0,0)[cc]{$u_3$}}
\put(35,29){\makebox(0,0)[cc]{$u_2$}}
\put(95,29){\makebox(0,0)[cc]{$u_2$}}
\put(50,23){\oval(16,20)[l]}
\put(110,23){\oval(16,20)[l]}
\put(50,20){\makebox(0,0)[cc]{$T^-$}}
\put(110,20){\makebox(0,0)[cc]{$T^-$}}
\put(45,25){\line(-1,0){10}}
\put(105,25){\line(-1,0){10}}
\put(35,25){\line(-1,0){10}}
\put(95,25){\line(-1,0){10}}
\qbezier(35,15)(25,11)(15,15)
\put(25,5){\makebox(0,0)[cc]
{$\sharp\alpha(T)=2\sharp\alpha^-_\in+4\sharp\alpha^-_{\not\in}$}}
\put(85,5){\makebox(0,0)[cc]
{$\sharp\alpha(T')=3\sharp\alpha^-_\in+5\sharp\alpha^-_{\not\in}$}}
\put(55,5){\makebox(0,0)[cc]{$<$}}
\put(25,33){\makebox(0,0)[cc]{$T$}}
\put(85,33){\makebox(0,0)[cc]{$T'$}}
\put(75,25){\line(1,0){10}}
\end{picture}\\[5mm]
\linethickness{0.4pt}
\ifx\plotpoint\undefined\newsavebox{\plotpoint}\fi 
\begin{picture}(118,33)(0,0)
\put(5,15){\circle*{1.5}}
\put(65,15){\circle*{1.5}}
\put(15,15){\circle*{1.5}}
\put(75,15){\circle*{1.5}}
\put(25,15){\circle*{1.5}}
\put(85,15){\circle*{1.5}}
\put(35,15){\circle*{1.5}}
\put(95,15){\circle*{1.5}}
\put(5,25){\circle*{1.5}}
\put(65,25){\circle*{1.5}}
\put(15,25){\circle*{1.5}}
\put(75,25){\circle*{1.5}}
\put(25,25){\circle*{1.5}}
\put(85,25){\circle*{1.5}}
\put(35,25){\circle*{1.5}}
\put(95,25){\circle*{1.5}}
\put(45,25){\circle*{1.5}}
\put(105,25){\circle*{1.5}}
\put(5,25){\line(0,-1){10}}
\put(65,25){\line(0,-1){10}}
\put(15,25){\line(0,-1){10}}
\put(75,25){\line(0,-1){10}}
\put(25,25){\line(0,-1){10}}
\put(85,25){\line(0,-1){10}}
\put(35,25){\line(0,-1){10}}
\put(95,25){\line(0,-1){10}}
\put(47,28){\makebox(0,0)[cc]{$u_3$}}
\put(107,28){\makebox(0,0)[cc]{$u_3$}}
\put(35,29){\makebox(0,0)[cc]{$u_2$}}
\put(95,29){\makebox(0,0)[cc]{$u_2$}}
\put(50,23){\oval(16,20)[l]}
\put(110,23){\oval(16,20)[l]}
\put(50,20){\makebox(0,0)[cc]{$T^-$}}
\put(110,20){\makebox(0,0)[cc]{$T^-$}}
\put(45,25){\line(-1,0){10}}
\put(105,25){\line(-1,0){10}}
\put(35,25){\line(-1,0){10}}
\put(95,25){\line(-1,0){10}}
\qbezier(35,15)(25,11)(15,15)
\qbezier(5,15)(20,7)(35,15)
\put(85,25){\line(-1,0){20}}
\put(25,5){\makebox(0,0)[cc]
{$\sharp\alpha(T)=2\sharp\alpha^-_\in+6\sharp\alpha^-_{\not\in}$}}
\put(85,5){\makebox(0,0)[cc]
{$\sharp\alpha(T')=5\sharp\alpha^-_\in+8\sharp\alpha^-_{\not\in}$}}
\put(55,5){\makebox(0,0)[cc]{$<$}}
\put(25,33){\makebox(0,0)[cc]{$T$}}
\put(85,33){\makebox(0,0)[cc]{$T'$}}
\end{picture}
\end{center}
\caption{$(d_1,d_2)\not\in \{ (0,1),(0,2),(1,1),(1,2)\}.$}\label{figexcl}
\end{figure}
Our next goal is to exclude the first four of these possible values of $(d_1,d_2)$.
In each case, we construct a subcubic tree $T'$ of order $n$
and independence number $\alpha=\frac{n}{2}$
such that $\sharp\alpha(T')>\sharp\alpha(T)$,
contradicting the choice of $T$.
Let $T^-=T-\bigcup_{e\in N_{\tilde{T}}(e_2)\setminus \{ e_3\}}e$.
By construction, the tree $T^-$ still has a perfect matching,
which implies $\alpha(T^-)=\frac{n(T^-)}{2}$.

Let 
\begin{itemize}
\item $\sharp\alpha^-_\in$ be the number of maximum independent sets in $T^-$ 
that contain $u_3$, and let 
\item $\sharp\alpha^-_{\not\in}$ be the number of maximum independent sets in $T^-$ 
that do not contain $u_3$.
\end{itemize}
Since $\alpha(T^-)=\frac{n(T^-)}{2}$, 
arguing as above implies that both partite sets of the bipartite graph $T^-$ 
are maximum independent sets in $T^-$ ,
which implies $\sharp\alpha^-_\in, \sharp\alpha^-_{\not\in}>0$.
Figure \ref{figexcl} illustrates the construction of $T'$ in each case,
together with the values of $\sharp\alpha(T)$ and $\sharp\alpha(T')$.

We conclude that $(d_1,d_2)=(1,0)$,
which implies that the subcubic tree $T'$ 
has order $n-4$ and independence number
$\alpha-2=\frac{n-4}{2}$.
Let $T''=T-\{ u_1,v_1\}$.
The subcubic tree $T''$ 
has order $n-2$ and independence number
$\alpha-1=\frac{n-2}{2}$.
Therefore, by the choice of $n$, we obtain
\begin{eqnarray}
\sharp\alpha(T) & = & 2\sharp\alpha^-_\in+3\sharp\alpha^-_{\not\in}\nonumber\\
& = & 
\Big(\sharp\alpha^-_\in+2\sharp\alpha^-_{\not\in}\Big)+
\Big(\sharp\alpha^-_\in+\sharp\alpha^-_{\not\in}\Big)\nonumber\\
& = & 
\sharp\alpha(T'')+\sharp\alpha(T')\nonumber\\
&\leq & 
f(\alpha-1+2)+f(\alpha-2+2)\label{e10}\\
&=& 
f(\alpha+2),\nonumber
\end{eqnarray}
that is, $\sharp\alpha(T)\leq f(\alpha+2)$.
Furthermore,
if $\sharp\alpha(T)=f(\alpha+2)$,
then equality holds in (\ref{e10}),
which, by the choice of $n$, implies 
$T'=T(\alpha-2)$ and $T''=T(\alpha-1)$,
and, hence, $T=T(\alpha)$.
This contradiction completes the proof. 
\end{proof}

\begin{proof}[Proof of Theorem \ref{theorem3}]
Suppose, for a contradiction, that the theorem is false,
and let $n$ be the smallest order for which it fails.
Let $T$ be a subcubic tree of order $n$
and independence number $\alpha$.
Let $u$ be an endvertex of a longest path $P$ in $T$.
By the choice of $n$, the path $P$ has order at least $3$.
Let $v$ be the neighbor of $u$,
and let $w$ be the neighbor of $v$ on $P$ that is distinct from $u$.
The subcubic tree $T'=T-(N_T[v]\setminus \{ w\})$
has order $n-d_T(v)$
and independence number $\alpha-(d_T(v)-1)$.
By the choice of $n$, we obtain
\begin{eqnarray}
\alpha & = & \alpha(T')+(d_T(v)-1)\nonumber\\
& \leq & \frac{2n(T')+1}{3}+(d_T(v)-1)\label{e12}\\
& = & \frac{2(n-d_T(v))+1}{3}+(d_T(v)-1)\nonumber\\
& = & \frac{2n+1}{3}-\frac{3-d_T(v)}{3}\nonumber\\
& \leq & \frac{2n+1}{3}\label{e13},
\end{eqnarray}
which implies (\ref{e11}).
Now, equality in (\ref{e11}) implies 
equality in (\ref{e12}) and (\ref{e13}).
By the choice of $n$,
the tree $T'$ arises from $K_1$ by iteratively attaching $P_3$s,
and that $v$ has degree $3$.
Hence, also $T$ arises from $K_1$ by iteratively attaching $P_3$s.
The uniqueness of the maximum independent set
follows easily by an inductive argument 
exploiting the constructive characterization of $T$.
This completes the proof. 
\end{proof}

\begin{proof}[Proof of Theorem \ref{theorem4}]
Suppose, for a contradiction, that the theorem is false, and let $n$ be the smallest order for which it fails. Let $T$ be a subcubic tree of order $n$ and independence number $\alpha$ such that $\sharp\alpha(T)$ is as large as possible.

\setcounter{claim}{0}

\begin{claim}\label{claim2}
The tree $T$ contains a path of length at least $3$.
\end{claim}
\begin{proof}[Proof of Claim \ref{claim2}]
Suppose, for a contradiction, that $T$ is a star $K_{1,n-1}$.

If $n=1$, then 
$$\sharp\alpha(T)=1=\left(\frac{1+\sqrt{5}}{2}\right)^{2-3+1},$$
if $n=2$, then 
$$\sharp\alpha(T)=2<2.618\approx\left(\frac{1+\sqrt{5}}{2}\right)^{4-3+1},$$
if $n=3$, then
$$\sharp\alpha(T)=1<1.618\approx\left(\frac{1+\sqrt{5}}{2}\right)^{6-6+1},$$
and, if $n=4$, then
$$\sharp\alpha(T)=1=\left(\frac{1+\sqrt{5}}{2}\right)^{8-9+1}.$$
In each case, we obtain a contradiction to the choice of $n$ and $T$.
\end{proof}
Let $uvwx\ldots r$ be a longest path in $T$, and consider $T$ as rooted in $r$. 
For a vertex $z$ of $T$, 
let $V_z$ be the set that contains $z$ and all its descendants.

\begin{claim}\label{claim3}
$d_T(v)=2$
\end{claim}
\begin{proof}[Proof of Claim \ref{claim3}]
Suppose, for a contradiction, that $d_T(v)=3$.
Note that every maximum independent set in $T$ 
contains both children of $v$ but not $v$.
Hence, 
the subcubic tree $T'=T-V(T_v)$ 
has order $n-3$ and independence number $\alpha-2$,
and satisfies $\sharp\alpha(T)=\sharp\alpha(T')$.
By the choice of $n$, we obtain
$$\sharp\alpha(T)=\sharp\alpha(T')\leq\left(\frac{1+\sqrt{5}}{2}\right)^{2\cdot(n-3)-3\cdot(\alpha-2)+1}=\left(\frac{1+\sqrt{5}}{2}\right)^{2n-3\alpha+1},$$
which contradicts the choice of $T$.
\end{proof}

\begin{claim}\label{claim4}
$w$ is not a support vertex.
\end{claim}
\begin{proof}[Proof of Claim \ref{claim4}]
Suppose, for a contradiction, that $w$ is a support vertex.
The subcubic tree $T'=T-V(T_v)$
has order $n-2$ and independence number $\alpha-1$, 
while the subcubic tree $T''=T-V(T_w)$
has order $n-4$ and independence number $\alpha-2$. 
Since there are 
$\sharp\alpha(T')$ maximum independent sets in $T$ that contain $u$,
and 
$\sharp\alpha(T'')$ maximum independent sets in $T$ that do not contain $u$,
the choice of $n$ implies
\begin{eqnarray*}
\sharp\alpha(T)&=&\sharp\alpha(T')+\sharp\alpha(T'')\\
&\leq &  \left(\frac{1+\sqrt{5}}{2}\right)^{2\cdot(n-2)-3\cdot(\alpha-1)+1}+ \left(\frac{1+\sqrt{5}}{2}\right)^{2\cdot(n-4)-3\cdot(\alpha-2)+1}\\
&=& \left(\frac{1+\sqrt{5}}{2}\right)^{2n-3\alpha+1}\left( \left(\frac{1+\sqrt{5}}{2}\right)^{-1}+ \left(\frac{1+\sqrt{5}}{2}\right)^{-2}\right)\\
&=&  \left(\frac{1+\sqrt{5}}{2}\right)^{2n-3\alpha+1},
\end{eqnarray*}
which contradicts the choice of $T$.
\end{proof}

\begin{claim}\label{claim5}
$d_T(w)=2$.
\end{claim}
\begin{proof}[Proof of Claim \ref{claim5}]
Suppose, for a contradiction, that $w$ has a child $v'$ distinct from $v$.
By Claims \ref{claim3} and \ref{claim4}, 
the vertex $v'$ has exactly one child $u'$,
which is an endvertex.
The subcubic tree $T'=T-\{u,v,u',v'\}$ 
has order $n-4$ and independence number $\alpha-2$. 
Since for every maximum independent set $I'$ of $T'$ that does not contain $w$,
we have $x\in I'$, and $(I'\setminus \{ x\})\cup \{ w\}$ is a 
maximum independent set in $T'$ that contains $w$,
there are at most $\frac{\sharp\alpha(T')}{2}$ maximum independent sets in $T'$ 
that do not contain $w$, and 
at least $\frac{\sharp\alpha(T')}{2}$ maximum independent sets in $T'$ 
that contain $w$. 
A maximum independent set in $T'$ that contains $w$ 
can only be extended in a unique way to a maximum independent set in $T$, 
while a maximum independent set in $T'$ that does not contain $w$ 
can be extended in four different ways to a maximum independent set in $T$. 
Since all maximum independent sets in $T$ are of one of these types,
the choice of $n$ implies
\begin{eqnarray*}
\sharp\alpha(T)&\leq&4\cdot\frac{\sharp\alpha(T')}{2}+\frac{\sharp\alpha(T')}{2}\\
&\leq &  \frac{5}{2}\cdot\left(\frac{1+\sqrt{5}}{2}\right)^{2\cdot(n-4)-3\cdot(\alpha-2)+1}\\
&=&  \frac{5}{2}\cdot\left(\frac{1+\sqrt{5}}{2}\right)^{-2} \left(\frac{1+\sqrt{5}}{2}\right)^{2n-3\alpha+1}\\
&<&  \left(\frac{1+\sqrt{5}}{2}\right)^{2n-3\alpha+1},
\end{eqnarray*}
using $\frac{5}{2}<\left(\frac{1+\sqrt{5}}{2}\right)^{2}$,
which contradicts the choice of $T$.
\end{proof}
Since $\sharp\alpha(P_4)=3<\left(\frac{1+\sqrt{5}}{2}\right)^{2\cdot 4-3\cdot2+1}$ we may assume that $x$ has a parent $y$. 

\begin{claim}\label{claim6}
$x$ is not a support vertex.
\end{claim}
\begin{proof}[Proof of Claim \ref{claim6}]
Suppose, for a contradiction, that $x$ has a child $w'$ that is an endvertex. 
The subcubic tree $T'=T-\{u,v,w\}$ 
has order $n-3$ and independence number $\alpha-2$. 
Every maximum independent set $I$ of $T$ contains $u$, $w$, and $w'$,
and $I\setminus \{ u,w\}$ is a maximum independent set in $T'$.
By the choice of $n$, this implies
$$\sharp\alpha(T)\leq\sharp\alpha(T')\leq\left(\frac{1+\sqrt{5}}{2}\right)^{2\cdot(n-3)-3\cdot(\alpha-2)+1}=\left(\frac{1+\sqrt{5}}{2}\right)^{2n-3\alpha+1},$$
which contradicts the choice of $T$.
\end{proof}

\begin{claim}\label{claim7}
$x$ has no child that is a support vertex.
\end{claim}
\begin{proof}[Proof of Claim \ref{claim7}]
Suppose, for a contradiction, that $x$ has a child $w'$ that is a support vertex.
If $w'$ has two children that are endvertices,
then arguing as in the proof of Claim \ref{claim3} yields a contradiction.
If $w'$ has a child that is not an endvertex,
then $d_T(w')=3$, 
which leads to a similar contradiction as in the proof of Claim \ref{claim5}.
Hence, $w'$ has a unique child $v'$, which is an endvertex.
The subcubic tree $T'=T-V(T_x)$ 
has order $n-6$ and independence number $\alpha-3$. 
A maximum independent set $I'$ of $T'$ can be extended 
in at most four different ways to a maximum independent set in $T$: 
$I'\cup\{u,v',x\}$, $I'\cup\{v,v',x\}$, $I'\cup\{u,w,w'\}$ and $I'\cup\{u,v',w\}$. 
Since all maximum independent sets in $T$ are of such a form,
the choice of $n$ implies
$$\sharp\alpha(T)\leq 4 \sharp\alpha(T')\leq 4\left(\frac{1+\sqrt{5}}{2}\right)^{2\cdot(n-6)-3\cdot(\alpha-3)+1}<\left(\frac{1+\sqrt{5}}{2}\right)^{2n-3\alpha+1},$$
using $4<\left(\frac{1+\sqrt{5}}{2}\right)^{3}$,
which contradicts the choice of $T$.
\end{proof}

\begin{claim}\label{claim8}
$d_T(x)=2$.
\end{claim}
\begin{proof}[Proof of Claim \ref{claim8}]
Suppose, for a contradiction, 
that $x$ has a child $w'$ distinct from $w$. 
By Claims \ref{claim6} and \ref{claim7},
$w'$ has a child $v'$ that has a child $u'$.
By Claims \ref{claim3} and \ref{claim5}, 
$d_T(w')=d_T(v')=2$.
The subcubic tree $T'=T-V(T_x)$
has order $n-7$ and independence number $\alpha-4$. 
Note that every maximum independent set in $T'$ 
can be extended in a unique way to a maximum independent set in $T$, 
and that the maximum independent sets in $T$ are exactly those sets. 
Hence, by the choice of $n$, we obtain 
$$\sharp\alpha(T)\leq \sharp\alpha(T')\leq\left(\frac{1+\sqrt{5}}{2}\right)^{2\cdot(n-7)-3\cdot(\alpha-4)+1}<\left(\frac{1+\sqrt{5}}{2}\right)^{2n-3\alpha+1}.$$
\end{proof}
By the above claims, we know that $d_T(v)=d_T(w)=d_T(x)=2$. 
Let 
$T'=T-V(T_x)$,
$T_1=T-\{vu\}+\{xu\}$, 
and $T''=T_1-\{v,w\}$. 
Clearly, all these trees are subcubic.

A maximum independent set in $T'$ that contains $y$ 
can only be extended in a unique way 
to a maximum independent set in $T$, 
and all maximum independent set in $T$ that contain $y$ are of that form. 
A maximum independent set $I'$ of $T'$ that does not contain $y$ 
can be extended to a maximum independent set $I$ of $T$ in three ways, $I'\cup\{u,w\}$, $I'\cup\{u,x\}$ and, $I'\cup\{v,x\}$, 
and every maximum independent set in $T$ that does not contain $y$ 
is of that form.

Similarly,
a maximum independent set in $T'$ that contains $y$ 
can be extended to a maximum independent set in $T_1$ 
in two different ways,
and all maximum independent set in $T_1$ that contain $y$ are of that form. 
A maximum independent set $I'$ of $T'$ that does not contain $y$ 
can be extended to a maximum independent set $I_1$ of $T_1$ in three ways, $I'\cup\{u,w\}$, $I'\cup\{u,v\}$ and, $I'\cup\{v,x\}$, 
and every maximum independent set in $T_1$ that does not contain $y$ 
is of that form. 
Arguing as in the proof of Claim \ref{claim4}, we obtain
$$\sharp\alpha(T)\leq\sharp\alpha(T_1)=\sharp\alpha(T')+\sharp\alpha(T'')\leq\left(\frac{1+\sqrt{5}}{2}\right)^{2n-3\alpha+1}.$$
This final contradiction completes the proof.
\end{proof}

\end{document}